\documentclass[11pt,reqno]{amsart}
\usepackage{caption}
 \usepackage{amssymb,amsfonts,amscd,amsbsy, color, amsmath}
\usepackage[mathscr]{eucal}
\usepackage{url}
\usepackage{graphicx}
\usepackage{mathtools}
\usepackage{todonotes}
\usepackage{tabularx}
\usepackage{tabu}
\usepackage{float, color}
\usepackage{placeins}
\usepackage{tikz}
\usepackage{cite}
\usetikzlibrary{decorations.pathmorphing}

\newtheorem{theorem}{Theorem}[section]
\newtheorem{lemma}[theorem]{Lemma}
\newtheorem{proposition}[theorem]{Proposition}
\newtheorem{corollary}[theorem]{Corollary}

\newtheorem{conjecture}[theorem]{Conjecture}
\theoremstyle{definition}
\newtheorem{remark}[theorem]{Remark}
\numberwithin{equation}{section}

\newif\ifshowverify
\showverifytrue
\ifshowverify

\else
\NewEnviron{verify}{\iffalse\expandafter\BODY\fi}
\fi

\newcommand{\Z}{\mathbb{Z}}
\newcommand{\Q}{\mathbb{Q}}

\newcommand{\C}{\mathbb{C}}

\usepackage{hyperref}
\hypersetup{
  colorlinks,
  linkcolor={blue},
  citecolor={cyan},
  urlcolor={blue}
}

\makeatletter
\def\imod#1{\allowbreak\mkern5mu({\operator@font mod}\,\,#1)}
\makeatother
\allowdisplaybreaks

\makeatletter
\@namedef{subjclassname@2020}{%
  \textup{2020} Mathematics Subject Classification}
\makeatother

\begin{document}

\title[Resurgence, Habiro elements and strange identities]{Resurgence, Habiro elements and strange identities}

\author[S. Crew]{Samuel Crew}

\author[V. Fantini]{Veronica Fantini}

\author[A. Goswami]{Ankush Goswami}

\author[R. Osburn]{Robert Osburn}

\author[C. Wheeler]{Campbell Wheeler}

\address{Department of Mathematics, Imperial College London, South Kensington Campus, London SW7 2AZ, UK} 
\email{samuel.crew24@imperial.ac.uk}

\address{Laboratoire Math{\'e}matique Orsay, Building 307, rue Michel Magat, Facult{\' e} des Sciences d’Orsay, Universit{\'e} Paris-Saclay, F-91405 Orsay Cedex, France}
\email{veronica.fantini@universite-paris-saclay.fr }

\address{School of Mathematical and Statistical Sciences, the University of Texas Rio Grande Valley, 1201 W. University Dr., Edinburg, TX, 78539}
\email{ankush.goswami@utrgv.edu}
\email{ankushgoswami3@gmail.com}

\address{School of Mathematics and Statistics, University College Dublin, Belfield, Dublin 4, Ireland}
\email{robert.osburn@ucd.ie}

\address{Theoretical Sciences Visiting Program, Okinawa Institute of Science and Technology Graduate University, Onna, 904-0495, Japan}

\address{IH\'ES, Le Bois-Marie, 35, route de Chartres, 91440 Bures-sur-Yvette, France}
\email{wheeler@ihes.fr }


\date{\today}

\begin{abstract}
We prove resurgence properties for the Borel transform of a formal power series associated to elements in the Habiro ring that come from radial limits of partial theta series via strange identities. As an application, we prove a conjecture in quantum topology due to Costin and Garoufalidis for two families of torus knots.
\end{abstract}

\maketitle

\section{Introduction}

The modern theory of resurgence began with the voluminous works of {\'E}calle in 1981 \cite{e1, e2} and 1985 \cite{e3}. This theory now plays a vital r{\^o}le in a remarkable number of diverse areas, to name just a few examples: nonlinear systems of ODEs and difference equations \cite{costin1, kamimoto}, algebraic combinatorics \cite{cfmt1, cfmt2}, enumerative combinatorics and quantum field theory \cite{BD, MY}, period integrals and string theory scattering amplitudes \cite{dk}, wall-crossing phenomena \cite{ks} and matrix models \cite{mm}. In this paper, we are interested in a conjecture due to Costin and Garoufalidis \cite{cg} which connects resurgence and quantum topology.  For other such interactions, see \cite{am,ggm,ggm2,ggmw,fw} and the references therein. 

Let us recall the general setup from \cite{cg}. For further details, see \cite{costin2, dp, bm, ms} or \cite[Section 2]{Sauzin} for an excellent concise review of the basic notions of resurgence and alien calculus. A formal power series 
\begin{equation} \label{fps}
\mathcal{F}(x)=\sum_{n=0}^{\infty} a_n x^{-n} \in \mathbb{C}[[1/x]]
\end{equation}
is called Gevrey-1 if there exists positive real numbers $A$, $B$ such that 
\begin{equation*}
| a_n | \leq AB^n n!
\end{equation*}
for all $n \geq 0$. Consider a Gevrey-1 formal power series given by (\ref{fps}) and its Borel transform $\mathcal{B}: \mathbb{C}[[1/x]] \longrightarrow \mathbb{C}[[p]]$ defined by
\begin{equation} \label{borel}
\mathcal{B} \Biggl( \sum_{n=0}^{\infty} a_n x^{-n} \Biggr) = a_0\delta+\sum_{n=1}^{\infty} a_{n} \frac{p^{n-1}}{(n-1)!} =: a_0\delta + G(p)
\end{equation}
where $\delta$ is a formal symbol that has Laplace transform given by the constant function $1$. For simplicity, we write $\mathcal{B}[\mathcal{F}](p)$ for the transform \eqref{borel}. Here, $G(p)$ has a positive radius of convergence as it arose from a Gevrey-1 power series. In a standard abuse of notation, $G(p)$ denotes the formal power series in (\ref{borel}) (and is also commonly referred to in the literature as the Borel transform of $\mathcal{F}(x)$) and the analytic continuation of the associated germ at the origin. In addition, $G(p)$ is called a \textit{resurgent function} if it is endlessly analytically continuable~\cite{ms}. This property means $G(p)$ extends to a (possibly multi-valued) holomorphic function along unbounded paths that need only circumvent a discrete set of singularities. In the present work, we consider the case where $G(p)$ defines a multi-valued function on $\mathbb{C} \setminus \mathcal{N}$, with $\mathcal{N}$ a discrete set of singularities lying on the ray $\mathbb{R}_{>0}$, e.g., $\mathcal{N} = \left \{ \frac{n^2\pi^2}{6}:n\in\Z_{>0} \right \}$.

Assuming $G(p)$ satisfies suitable growth conditions, the left and right Borel resummations of $\mathcal{F}(x)$ are defined as 
\begin{eqnarray*}
S^{+}[\mathcal{F}](x) =\int_{\gamma_l}e^{-px} \mathcal{B}[\mathcal{F}](p) \,dp
\end{eqnarray*}
and
\begin{eqnarray*}
S^{-}[\mathcal{F}](x) =\int_{\gamma_r}e^{-px}  \mathcal{B}[\mathcal{F}](p) \, dp
\end{eqnarray*}
where $x \in \mathbb{C}$ and $\gamma_l$ and $\gamma_r$ are contours in $\mathbb{C} \setminus \mathcal{N}$ from $0$ to $\infty$ that turn left (respectively, right) at each singularity in $\mathcal{N}$ (see Figure~\ref{fig:fig1}). In addition, the left and right Borel resummations give analytic functions on a sector of opening angle $\pi$ in the complex $x$-plane.
\begin{figure}[ht]
    \centering

\tikzset{every picture/.style={line width=0.75pt}} 

\begin{tikzpicture}[x=0.75pt,y=0.75pt,yscale=-1,xscale=1]

\draw [color={rgb, 255:red, 80; green, 227; blue, 194 }  ,draw opacity=1 ]   (100.6,180) -- (170.6,180) ;
\draw [color={rgb, 255:red, 80; green, 227; blue, 194 }  ,draw opacity=1 ]   (60.4,180) .. controls (60.2,155.4) and (99.8,155.4) .. (100.6,180) ;
\draw [color={rgb, 255:red, 80; green, 227; blue, 194 }  ,draw opacity=1 ]   (170.6,180) .. controls (169.8,155.8) and (209.8,155.4) .. (210.8,180) ;
\draw [color={rgb, 255:red, 80; green, 227; blue, 194 }  ,draw opacity=1 ]   (210.8,180) -- (250.6,180.2) ;
\draw [color={rgb, 255:red, 80; green, 227; blue, 194 }  ,draw opacity=1 ]   (250.6,180.2) .. controls (250.2,156.2) and (290.2,156.2) .. (290.8,180.2) ;
\draw [color={rgb, 255:red, 80; green, 227; blue, 194 }  ,draw opacity=1 ]   (290.8,180.2) -- (330.6,180.4) ;
\draw [color={rgb, 255:red, 80; green, 227; blue, 194 }  ,draw opacity=1 ]   (20.6,179.8) -- (60.4,180) ;
\draw [color={rgb, 255:red, 74; green, 144; blue, 226 }  ,draw opacity=1 ]   (21,190) -- (60.8,190.2) ;
\draw [color={rgb, 255:red, 74; green, 144; blue, 226 }  ,draw opacity=1 ]   (101,190) -- (171,190) ;
\draw [color={rgb, 255:red, 74; green, 144; blue, 226 }  ,draw opacity=1 ]   (211.2,190) -- (250.2,190.2) ;
\draw [color={rgb, 255:red, 74; green, 144; blue, 226 }  ,draw opacity=1 ]   (291.2,190) -- (331,190.2) ;
\draw [color={rgb, 255:red, 74; green, 144; blue, 226 }  ,draw opacity=1 ]   (250.2,190.2) .. controls (249.4,215) and (290.2,215.4) .. (291.2,190) ;
\draw [color={rgb, 255:red, 74; green, 144; blue, 226 }  ,draw opacity=1 ]   (171,190) .. controls (170.6,214.2) and (210.2,215) .. (211.2,190) ;
\draw [color={rgb, 255:red, 74; green, 144; blue, 226 }  ,draw opacity=1 ]   (60.8,190.2) .. controls (60.6,215.8) and (100.2,214.6) .. (101,190) ;
\draw  [color={rgb, 255:red, 255; green, 0; blue, 0 }  ,draw opacity=1 ][fill={rgb, 255:red, 255; green, 0; blue, 0 }  ,fill opacity=1 ] (75.4,185.9) .. controls (75.4,182.64) and (78.04,180) .. (81.3,180) .. controls (84.56,180) and (87.2,182.64) .. (87.2,185.9) .. controls (87.2,189.16) and (84.56,191.8) .. (81.3,191.8) .. controls (78.04,191.8) and (75.4,189.16) .. (75.4,185.9) -- cycle ;
\draw  [color={rgb, 255:red, 255; green, 0; blue, 0 }  ,draw opacity=1 ][fill={rgb, 255:red, 255; green, 0; blue, 0 }  ,fill opacity=1 ] (185,185.7) .. controls (185,182.44) and (187.64,179.8) .. (190.9,179.8) .. controls (194.16,179.8) and (196.8,182.44) .. (196.8,185.7) .. controls (196.8,188.96) and (194.16,191.6) .. (190.9,191.6) .. controls (187.64,191.6) and (185,188.96) .. (185,185.7) -- cycle ;
\draw  [color={rgb, 255:red, 255; green, 0; blue, 0 }  ,draw opacity=1 ][fill={rgb, 255:red, 255; green, 0; blue, 0 }  ,fill opacity=1 ] (265.4,185.5) .. controls (265.4,182.24) and (268.04,179.6) .. (271.3,179.6) .. controls (274.56,179.6) and (277.2,182.24) .. (277.2,185.5) .. controls (277.2,188.76) and (274.56,191.4) .. (271.3,191.4) .. controls (268.04,191.4) and (265.4,188.76) .. (265.4,185.5) -- cycle ;
\draw  [color={rgb, 255:red, 80; green, 227; blue, 194 }  ,draw opacity=1 ][fill={rgb, 255:red, 80; green, 227; blue, 194 }  ,fill opacity=1 ] (235.06,180.06) -- (229.71,183.04) -- (229.7,177.1) -- cycle ;
\draw  [color={rgb, 255:red, 74; green, 144; blue, 226 }  ,draw opacity=1 ][fill={rgb, 255:red, 74; green, 144; blue, 226 }  ,fill opacity=1 ] (136,190) -- (130.65,192.98) -- (130.64,187.04) -- cycle ;

\draw (264.6,141.8) node [anchor=north west][inner sep=0.75pt]  [color={rgb, 255:red, 80; green, 227; blue, 194 }  ,opacity=1 ] [align=left] {$\displaystyle \gamma _{l}$};
\draw (264.8,212) node [anchor=north west][inner sep=0.75pt]  [color={rgb, 255:red, 74; green, 144; blue, 226 }  ,opacity=1 ] [align=left] {$\displaystyle \gamma _{r}$};
\draw (115.6,147) node [anchor=north west][inner sep=0.75pt]    {$\mathcal{\textcolor[rgb]{1,0,0}{N}}$};
\end{tikzpicture}
\caption{A singularity set $\mathcal{N} \subset \mathbb{C}$ and the contours $\gamma_l$ and $\gamma_r$.}
\label{fig:fig1}
\end{figure}

\noindent Following \cite[Sections 4.2--4.3]{cg}, the median resummation of $\mathcal{F}$ is given by
\begin{equation} \label{smed}
S^{\textbf{med}}(x) = S^{\textbf{med}}[\mathcal{F}](x) := \frac{1}{2} \Bigl( S^{+}[\mathcal{F}](x) + S^{-}[\mathcal{F}](x) \Bigr).
\end{equation}
In general, care is required when defining median summations, see \cite[Section 2]{hlss}. Under favorable circumstances, for example when $G(p)$ is the Borel transform of a formal (Gevrey-1) solution in powers of $1/x$ to a linear differential equation, each of $\mathcal{S}^{\pm}[\mathcal{F}](x)$ (and thus, of course, $S^{\textbf{med}}(x)$) is a well-defined homomorphic solution. This is one incarnation of the terminology “resummation”.

In this paper, we will study the median resummation of the Borel transform for certain formal power series associated to knots.

Let $K$ be a knot and $J_N(K; q)$ be the usual colored Jones polynomial, normalized to be 1 for the unknot. If $N=2$, then we recover the Jones polynomial \cite{Jones}. As a knot invariant, the colored Jones polynomial is of fundamental importance in several open problems in quantum topology (see, e.g., \cite{garoufalidis-AJ, garoufalidis-slope, kal-tran, my, z}). We now recall the Habiro ring \cite{habiro}
\begin{equation*}
\displaystyle \mathcal{H} := \varprojlim_{n} \mathbb{Z}[q] / \langle (q)_n \rangle
\end{equation*}
where 
\begin{equation*}
(a)_n = (a;q)_n := \prod_{k=1}^{n} (1-aq^{k-1})
\end{equation*}
is the standard $q$-Pochhammer symbol. Every element of $\mathcal{H}$ may be written as
\begin{equation*}
\sum_{n=0}^{\infty} b_n(q) (q)_n
\end{equation*}
where $b_n(q) \in \mathbb{Z}[q]$. Given $K$, there exists an element $\Phi_{K}(q) \in \mathcal{H}$ called the Kashaev invariant~\cite{Kashaev:6j} which matches $J_N(K; q)$ (up to a prefactor) when $q=\zeta_N := e^{\frac{2 \pi i}{N}}$ \cite[Theorem 2]{hl}. To illustrate, for the trefoil knot $K=3_1$, we have \cite{habiro1, thang}
\begin{equation}\label{cjp31}
J_{N}(3_1; q) = q^{1-N} \sum_{n=0}^{\infty} q^{-nN} (q^{1-N})_n.
\end{equation}
In this case, the Kontsevich-Zagier series \cite{z}
\begin{equation} \label{kz}
\Phi_{3_1}(q) := \sum_{n=0}^{\infty} (q)_n \in \mathcal{H}
\end{equation}
gives the Kashaev invariant for the trefoil~\cite{Kashaev:6j} and so from \eqref{cjp31} and \eqref{kz}, we observe
\begin{equation*}
\Phi_{3_1}(\zeta_N) \zeta_N = J_N(3_1; \zeta_N).
\end{equation*}
Now, a crucial aspect of the computations in \cite{cg} is the ``strange identity" 
\begin{equation} \label{strid}
\Phi_{3_1}(q) ``=" -\frac{1}{2} \sum_{n=1}^{\infty} n \Bigl( \frac{12}{n} \Bigr) q^{\frac{n^2 - 1}{24}}
\end{equation}
where $``="$ means that both sides agree to infinite order at any root of unity and $\bigl( \frac{12}{*} \bigr)$ is the quadratic character of conductor $12$. More precisely, replacing $q$ by the product of a root of unity and $e^{-x}$, then letting $x \to 0^{+}$, the right-hand side has an asymptotic expansion as a power series in $x$ and this power series is given by the left-hand side at $q$. For further details, see \cite{akl, z1}. If $q=e^{2 \pi i x}$, $x \in \mathbb{H}$, then (\ref{strid}) implies for $0 \neq \alpha \in \mathbb{Q}$
\begin{equation} \label{rational}
e^{\frac{\pi i \alpha}{12}} \Phi_{3_1}(e^{2 \pi i \alpha}) = -\frac{1}{2} \sum_{n=1}^{\infty} n \Bigl( \frac{12}{n} \Bigr) e^{\frac{n^2 \pi i \alpha}{12}}
\end{equation}
where the right-hand side of (\ref{rational}) is interpreted as the radial limit $x \to \alpha$. The significance of ``identities" akin to (\ref{strid}) is evident in various applications, e.g., asymptotics and congruences for Fishburn and generalized Fishburn numbers modulo prime powers \cite{ak, akl, as, Be, an, gjko, straub, z1}, the quantum modularity of $\Phi_{3_1}(q)$ \cite{go, z} and expressing WRT invariants of Brieskorn homology spheres in terms of limiting values of Eichler integrals \cite{hikamiWRT1, hikamiWRT2, hikamiWRT3}. For a recent advance using the ``Bailey machinery" which not only recovers (\ref{strid}), but produces a wealth of new examples, see \cite{lovejoy}.

Elements of $\mathcal{H}$ can also be formally expanded around any root of unity. Therefore, we let
\begin{equation} \label{habel}
F_{K}(x) := \Phi_{K}(e^{-\frac{1}{x}})\in\C[\![1/x]\!].
\end{equation}
To emphasize the dependence on $K$, we denote the Borel transform of $F_{K}(x)$ by $G_{K}(p)$ and its conjecturally well-defined median resummation by $S_{K}^{\textbf{med}}(x)$. Finally, we use the notation $\stackrel{.}{=}$ to denote equality up to a prefactor which depends on $K$. We can now state the main conjecture from \cite{cg} (slightly edited for clarity) which is part of a larger program to understand the analytic continuation of invariants of ``knotted objects" arising in Chern-Simons theory \cite{g, gl}.

\begin{conjecture}\label{cgconj} For every knot $K$,
\begin{itemize}
\item[(1)] $F_K(x)$ has a resurgent Borel transform $G_{K}(p)$. \\

\item[(2)] $S^{\textbf{med}}_{K}(x)$ is an analytic function defined on $\Re(x) > 0$ with radial limits at the points $\frac{1}{2\pi i} \mathbb{Q}$ of its natural boundary. \\

\item[(3)] If $(2)$ is true, then consider the radial limit
\begin{equation*}
S_{K}^{\bf{med}}\left( -\frac{1}{2 \pi i \alpha} \right)  := \lim_{x \to -\frac{1}{2 \pi i \alpha}} S_{K}^{\bf{med}}(x)
\end{equation*}
where $\Re(x) > 0$. For $0 \neq \alpha \in\Q$, we have $S_{K}^{\textbf{med}}\left(-\frac{1}{2 \pi i \alpha}\right) \stackrel{.}{=} \Phi_{K}(e^{2\pi i \alpha})$.
\end{itemize}
\end{conjecture}

In~\cite[Theorem 2.5, Theorem 3.1]{cg}, Conjecture~\ref{cgconj} (1) and (2) were proven for $K=3_1$. Here,
\begin{equation*}
F_{3_1}(x) = e^{-\frac{1}{24x}} \sum_{n=0}^{\infty} (e^{-\frac{1}{x}})_n
\end{equation*}
where an adjustment of (\ref{habel}) has been made. The Borel transform $G_{3_1}(p)$ of $F_{3_1}(x)$ is explicitly given by \cite[Theorem 3.1]{cg}
\begin{equation} \label{g31}
G_{3_1}(p) = \frac{3 \pi}{2\sqrt{2}} \sum_{n=1}^{\infty} \frac{n \bigl( \frac{12}{n} \bigr)}{\bigl (-p + \frac{n^2 \pi^2}{6} \bigr)^{5/2}}.
\end{equation}
Moreover, a close inspection of the proofs of Theorems 2.5 and 3.1 in \cite{cg} reveals that it is actually (\ref{strid}) and the periodic function $\bigl( \frac{12}{*} \bigr)$ which determine the analytic nature of $G_{3_1}(p)$ and $S^{\textbf{med}}_{3_1}(x)$. Our main goal in this paper is to prove resurgence properties for formal power series associated to elements in $\mathcal{H}$ that satisfy a general type of strange identity motivated by (\ref{strid}). Here, we emphasize the importance of periodic functions. Before stating our main result, we introduce some notation.

Let $M \geq 2$ be an integer and $k_1, k_2 \in \mathbb{Z}$ with $0 < k_1 < k_2 < \frac{M}{2}$. Let $0\neq c\in\mathbb{R}$ and $f$ be the function
\begin{eqnarray}\label{ggen}
f(n):=\begin{cases}
c &\text{if $n\equiv k_1,M-k_1\pmod{M}$,} \\
-c &\text{if $n\equiv k_2,M-k_2\pmod{M}$,}\\
0 &\text{otherwise.}
\end{cases} 
\end{eqnarray}
Note that under the conditions on $k_1$ and $k_2$, $f$ is a well-defined even function of period $M$ with mean value zero. For integers $a \geq 0$ and $b>0$, consider the partial theta series
\begin{equation} \label{pts}
\theta_{a,b,f}^{(\nu)}(q):=\sum_{n=0}^{\infty} n^\nu f(n)q^{\frac{n^2-a}{b}}
\end{equation}
where $q=e^{2\pi i x}$, $x \in \mathbb{H}$, and $\nu \in \{0, 1\}$. For $\ell \in \mathbb{Z}$, set
\begin{equation} \label{ftilde}
\tilde{f}(\ell) := (-1)^{\ell} \sin\left(\frac{(k_2-k_1)\ell\pi}{M}\right)\sin\left(\frac{(M-k_1-k_2)\ell\pi}{M}\right).
\end{equation}
Suppose there exists
\begin{equation*}
\Phi_{f}(q) := \sum_{n=0}^{\infty} A_{n,f}(q) (q)_n \in \mathcal{H}
 \end{equation*}
where $A_{n,f}(q) \in \mathbb{Z}[q]$ such that
\begin{equation} \label{Hstrange}
 \Phi_{f}(q) ``=" \theta_{a,b,f}^{(1)}(q).
\end{equation}
Finally, define 
\begin{equation} \label{gen-expansions2}
\mathcal{F}_{f} \left( x \right) := e^{-\frac{a}{bx}} {\Phi}_f(e^{-\frac{1}{x}})=\sum_{n=0}^\infty\dfrac{C_{n,f}}{n!}\left(\frac{1}{bx}\right)^n\,.
\end{equation}
The asymptotics of the partial theta series $\theta_{a,b,f}^{(1)}(q)$ around $q=1$, which by the strange identity in \eqref{Hstrange} gives the expansion of $\Phi_f(q)$, can be computed by $L$-values. By standard calculations using the Mellin transform (see, e.g., \cite{akl, gjko, z}), we find that
\begin{eqnarray*}\label{Cnf}
C_{n,f}=(-1)^nL(-2n-1,f)
\end{eqnarray*}
where
\begin{eqnarray*}
 L(s,f) :=\sum_{n=1}^\infty\dfrac{f(n)}{n^s}
\end{eqnarray*}
is the $L$-series associated to $f$. In particular, using \cite[Lemma 3.2]{akl}, we obtain that
\begin{eqnarray}\label{CnfBer}
C_{n,f}=(-1)^{n+1} \dfrac{M^{2n+1}}{2n+2} \sum_{m=1}^M f(m)B_{2n+2}\left(\dfrac{m}{M}\right)
\end{eqnarray}
where for $k\geq 0$, $B_k(x)$ denotes the $k$th Bernoulli polynomial. By (\ref{gen-expansions2}) and \eqref{CnfBer}, we have
\begin{equation} \label{F}
\mathcal{F}_{f} \left( \frac{1}{x} \right)  = \sum_{n=0}^{\infty} \frac{(-1)^{n+1} M^{2n+1}}{(2n+2) n!} \sum_{m=1}^{M} f(m) B_{2n+2}\left(\dfrac{m}{M}\right) \left(\frac{x}{b}\right)^n.
\end{equation}
We will denote the Borel transform of $\mathcal{F}_f(x)$ by $G_f(p)$. Here, we write $S_{f}^{\bf{med}}(x)$ for $S^{\textbf{med}}(x)$, \begin{equation*}
S_{f}^{\bf{med}}\left( -\frac{1}{2 \pi i \alpha} \right)  := \lim_{x \to -\frac{1}{2 \pi i \alpha}} S_{f}^{\bf{med}}(x),
\end{equation*} 
$\theta_{a,b,f}^{(\nu)}(x)$ for $\theta_{a,b,f}^{(\nu)}(q)=\theta_{a,b,f}^{(\nu)}(e^{2\pi i x})$ and $\displaystyle \theta_{a,b,f}^{(\nu)}(\alpha) := \lim_{x \to \alpha} \theta_{a,b,f}^{(\nu)}(x)$ where $0 \neq \alpha \in \mathbb{Q}$. Our main result is now the following.

\begin{theorem}\label{main} For $f$ given by \eqref{ggen} and $\mathcal{F}_f(x)$ as in \eqref{F}, we have the following:
\begin{itemize}
\item[(1)] $\mathcal{F}_f(x)$ has a resurgent Borel transform $G_f(p)$. \\

\item[(2)] $S_f^{\textbf{med}}(x)$ is an analytic function defined on $\Re(x) > 0$ with radial limits at the points $\frac{1}{2\pi i} \mathbb{Q}$ of its natural boundary. \\

\item[(3)] For $0 \neq \alpha\in\mathbb{Q}$, we have
\begin{equation}\label{eq:smed.form}
S^{\textbf{med}}_f\left(-\dfrac{1}{2\pi i\alpha}\right)=\frac{c b e^{\pi i/4}}{M\pi (i\alpha)^{3/2}}\int_0^{+i\infty} \frac{\theta_{0,4M^2,\tilde{f}}^{(0)}(bp)}{\big(\frac{1}{\alpha}+ p\big)^{3/2}} dp+\left(\frac{b}{i\alpha}\right)^{3/2}\frac{\sqrt{2} c}{M^2}\theta^{(1)}_{0,4M^2,\tilde{f}}\left(-\frac{b}{\alpha}\right).
\end{equation}
\end{itemize}
\end{theorem}

\begin{remark} 
The proof of Theorem \ref{main} demonstrates resurgence properties for the normalized partial theta series $e^{-\frac{a}{bx}} \theta_{a,b,f}^{(1)}(e^{-\frac{1}{x}})$ without assuming \eqref{Hstrange}. Our results and those in the announcement \cite{hlss} appear to be related; however, we give a proof that can be used to prove Conjecture 1.1 for our cases of interest (see Corollary \ref{cor}).
\end{remark}

We give two applications of Theorem \ref{main}. The first result explicitly computes \eqref{eq:smed.form} in the following situation. For coprime positive integers $s$ and $t$, $1\le m\le t-1$ and $1\le n\le s-1$, let $\chi_{s,t}^{(n,m)}(k)$ be the periodic function obtained by choosing $c=1$, $M=2st$, $k_1= | nt-ms |$ and $k_2 = nt+ms$ in \eqref{ggen}, i.e.,
\begin{align}\label{pf}
\chi_{s,t}^{(n,m)}(k)=\begin{cases}
1 & \text{if $k\equiv \pm(nt-ms)\pmod{2st}$},\\
-1& \text{if $k\equiv \pm(nt+ms)\pmod{2st}$},\\
0 &\text{otherwise}.
\end{cases}
\end{align}
For coprime odd integers $s$, $t$, define the sets
\begin{align} \label{d1}
\mathscr{D}_1(s, t)&:=\left\{(n, m): 1\le n\le \frac{s-1}{2}, 1\le m\le t-1\right\}
\end{align}
and
\begin{align} \label{d2}
\mathscr{D}_2(s, t)&:=\left\{(n, m): 1\le n\le s-1, 1\le m\le \frac{t-1}{2}\right\}.
\end{align}
There is a bijection between $\mathscr{D}_1(s, t)$ and $\mathscr{D}_2(s, t)$ (see Lemma \ref{scha} and \eqref{bi}). For $s$ and $t$ coprime and of opposite parities, we consider
\begin{align*}
\mathscr{D}_3(s, t):=
\begin{cases}
\left\{(n, m): 1\le n\le \frac{s-1}{2}, 1\le m\le t-1\right\} & \text{if $s\equiv 1\!\!\!\pmod{2}$},\\
\left\{(n, m): 1\le n\le s-1, 1\le m\le \frac{t-1}{2}\right\} & \text{if $s\equiv 0\!\!\!\pmod{2}$}.
\end{cases}
\end{align*}
Finally, define the set $\mathscr{D}(s, t)$ of pairs which yield distinct characters in (\ref{pf}) as follows:
\begin{align} \label{d}
\mathscr{D}(s, t):=
\begin{cases}
\mathscr{D}_1(s, t)\quad \text{or}\quad \mathscr{D}_2(s, t) & \text{if $s\equiv t\equiv 1\!\!\!\pmod{2}$},\\
\mathscr{D}_3(s, t) & \text{if $s\not\equiv t\!\!\!\pmod{2}$}.
\end{cases}    
\end{align}
One sees that $D(s, t):=|\mathscr{D}(s, t)|=\frac{(s-1)(t-1)}{2}$. 

\begin{theorem} \label{main2} Let $s$ and $t$ be coprime positive integers and $(n,m) \in \mathscr{D}(s,t)$. For $0 \neq \alpha \in \mathbb{Q}$, we have
\begin{equation}
S_{\chi_{s,t}^{(n,m)}}^{\bf{med}}\left(-\frac{1}{2\pi i\alpha}\right) = \theta_{0,4st,\chi_{s,t}^{(n,m)}}^{(1)}(\alpha).  
\end{equation}
\end{theorem}
The second result verifies Conjecture \ref{cgconj} for two families of torus knots.

\begin{corollary}\label{cor} Let $u,k\in\mathbb{Z}_{\geq 1}$. Conjecture \ref{cgconj} is true for the families of torus knots $T(2, 2u+1)$ and $T(3, 2^k)$.
\end{corollary}

Corollary \ref{cor} generalizes Theorems 2.5 and 3.1 in~\cite{cg} and upon taking either $u=1$ or $k=1$ confirms Conjecture~\ref{cgconj} for the trefoil knot $3_1 = T(2,3)=T(3,2)$. 

The paper is organized as follows. In Section~\ref{sec:proof-main-thm}, we prove Theorem~\ref{main} by first providing an explicit evaluation of $G_{f}(p)$ in \eqref{BT}, which generalizes \eqref{g31}, and then computing $S_{f}^{\textbf{med}}(x)$ directly via a careful contour deformation argument. In Section \ref{sec:torus}, we prove Theorem~\ref{main2} and Corollary~\ref{cor}. In Section~\ref{sec:remarks}, we make some concluding remarks.

\section{Proof of Theorem \ref{main}}\label{sec:proof-main-thm}

\begin{proof}[Proof of Theorem \ref{main}]
The Borel transform of $\mathcal{F}_f(1/x)$ from \eqref{F} is given by  
\begin{equation*}
\mathcal{B}[\mathcal{F}_{f}](p) = C_M \delta + G_f(p) 
\end{equation*}
where
\begin{equation*}
C_M = -\frac{M}{2} \sum_{m=1}^{M} f(m) B_{2}\left(\dfrac{m}{M}\right)
\end{equation*}
and
\begin{equation} \label{gfp}
\begin{aligned}
G_f(p) & = \sum_{n=1}^{\infty} \frac{(-1)^{n+1} M^{2n+1}}{(2n+2) n!} \sum_{m=1}^{M} f(m) B_{2n+2}\left(\dfrac{m}{M}\right) \frac{p^{n-1}}{(n-1)! b^n} \\
& = M \sum_{n=0}^{\infty} \sum_{m=1}^{M} \frac{(-1)^n f(m) B_{2n+4}\left(\dfrac{m}{M}\right)}{(2n+4)!} \frac{(2n+3)!}{n! (n+1)!} \left( \frac{M^2}{b} \right)^{n+1} p^n.
\end{aligned}
\end{equation}
From the generating function
\begin{eqnarray*}\label{Berpolygen}
\dfrac{xe^{xt}}{e^x-1}=\sum_{n=0}^\infty\dfrac{B_n(t)}{n!}x^n, 
\end{eqnarray*}
where $|x|<2\pi$, we deduce
\begin{eqnarray} \label{b}
\sum_{n=0}^{\infty} \dfrac{B_{2n}(t)}{(2n)!}x^{2n}=\dfrac{x}{2}\left(\dfrac{e^{xt}}{e^x-1}-\dfrac{e^{-xt}}{e^{-x}-1}\right)=\dfrac{x(e^{xt}+e^{-(t-1)x})}{2(e^{x}-1)}.\label{evenBer}   
\end{eqnarray}
Assuming $|y|<\frac{2\pi}{M}$, $x=iMy$, and taking $t=\frac{m}{M}$, where $1\leq m\leq M$, we use (\ref{b}) to obtain
\begin{align} \label{yep}
\dfrac{1}{My}\sum_{m=1}^Mf(m)\sum_{n=0}^\infty\dfrac{B_{2n+2}\left(\frac{m}{M}\right)}{(2n+2)!}\left(iMy\right)^{2n+2}&=\dfrac{i}{2(e^{iMy}-1)}\sum_{m=1}^{M} f(m)\left(e^{imy}+e^{-i(m-M)y} \right) \notag \\
&=-2c\cdot\dfrac{\sin\left(\frac{(k_2-k_1)y}{2}\right)\sin\left(\frac{(M-k_1-k_2)y}{2}\right)}{\sin\left(\frac{My}{2}\right)}.
\end{align}
If we let $y=\sqrt{p}$ in (\ref{yep}), then
\begin{align} 
\notag g_{1,f}(p)&:=M^{3} \sum_{n=0}^{\infty}\sum_{m=1}^M f(m)\dfrac{B_{2n+4}\left(\frac{m}{M}\right)}{(2n+4)!}\,  (-M^2p)^{n}\\
&= \frac{1}{Mp^2}\sum_{n=0}^{\infty}\sum_{m=1}^M f(m)\dfrac{B_{2n+2}\left(\frac{m}{M}\right)}{(2n+2)!}\,  (-M^2p)^{n+1} - \frac{C_M}{p}\notag\\
&\label{f1}= -\dfrac{1}{p^{3/2}}\left(2c\cdot\dfrac{\sin\left(\frac{(k_2-k_1)\sqrt{p}}{2}\right)\sin\left(\frac{(M-k_1-k_2)\sqrt{p}}{2}\right)}{\sin\left(\frac{M\sqrt{p}}{2}\right)} + C_M\sqrt{p}\right).
\end{align}
We conclude 
\begin{eqnarray*}
G_f(p)=(g_{1,f}\circledast g_{2,f})(p)    
\end{eqnarray*}
where $g_{1,f}$ was computed in \eqref{f1} and 
\begin{align*}
g_{2,f}(p)&:=\sum_{n=0}^{\infty} \dfrac{(2n+3)!}{n!(n+1)!}\Big(\dfrac{1}{b}\Big)^{n+1} p^{n} =\dfrac{1}{b}\dfrac{6}{ \left(1-4\frac{p}{b} \right)^{5/2}} 
\end{align*}
and $\circledast$ denotes the Hadamard product of two formal power series
\begin{eqnarray*}
\Biggl( \sum_{n=0}^{\infty} a_n p^n \Biggr) \circledast \Biggl( \sum_{n=0}^{\infty} b_n p^n \Biggr) := \sum_{n=0}^{\infty} a_n b_n p^n.
\end{eqnarray*}
It now follows that
\begin{eqnarray}\label{int}
G_f(p)=\dfrac{1}{2\pi i}\int_{-\gamma} g_{1,f}(s)g_{2,f}\left(\frac{p}{s}\right)\dfrac{ds}{s}    
\end{eqnarray}
where $\gamma$ is a circle with center at the origin and a small radius oriented clockwise (see Figure \ref{fig:fig2}). 
 \begin{figure}[h]
    \centering
\tikzset{every picture/.style={line width=0.75pt}} 

\begin{tikzpicture}[x=0.75pt,y=0.75pt,yscale=-1,xscale=1]

\draw [color={rgb, 255:red, 155; green, 155; blue, 155 }  ,draw opacity=1 ]   (200,40.17) -- (200,221) ;
\draw [color={rgb, 255:red, 155; green, 155; blue, 155 }  ,draw opacity=1 ]   (110.67,130.83) -- (290.33,130.17) ;
\draw    (210.5,81) .. controls (271.33,91.17) and (260.5,180.75) .. (199.5,180.75)(200,180.75) .. controls (139.5,180.75) and (131,91) .. (190,81) ;
\draw    (210.33,59.5) .. controls (300,69.83) and (300,200.5) .. (199.67,200.83) .. controls (99.33,201.17) and (101.33,70.83) .. (190,59.83) ;
\draw    (190,59.83) -- (190,81) ;
\draw    (210.33,59.5) -- (210.5,81) ;
\draw  [fill={rgb, 255:red, 0; green, 0; blue, 0 }  ,fill opacity=1 ] (162.33,95.56) -- (154.95,96.77) -- (162.76,103.03) -- cycle ;
\draw  [fill={rgb, 255:red, 0; green, 0; blue, 0 }  ,fill opacity=1 ] (190,67.63) -- (184.99,73.2) -- (195,73.2) -- cycle ;
\draw  [fill={rgb, 255:red, 0; green, 0; blue, 0 }  ,fill opacity=1 ] (132.03,165.45) -- (133.44,158.09) -- (124.87,163.27) -- cycle ;
\draw  [fill={rgb, 255:red, 0; green, 0; blue, 0 }  ,fill opacity=1 ] (267.39,95.88) -- (265.27,103.06) -- (274.31,98.76) -- cycle ;
\draw  [fill={rgb, 255:red, 0; green, 0; blue, 0 }  ,fill opacity=1 ] (210.42,73.25) -- (215.42,67.68) -- (205.41,67.68) -- cycle ;
\draw  [fill={rgb, 255:red, 0; green, 0; blue, 0 }  ,fill opacity=1 ] (235.74,166.06) -- (242.97,164.14) -- (234.59,158.66) -- cycle ;

\draw (125,182.4) node [anchor=north west][inner sep=0.75pt]    {$\gamma _{R}$};

\draw (160,150.4) node [anchor=north west][inner sep=0.75pt]    {$\gamma$};

\draw (220,176.4) node [anchor=north west][inner sep=0.75pt]    {};
\end{tikzpicture}

    \caption{Contour deformation of $\gamma_R - \gamma$}
    \label{fig:fig2}
\end{figure}

\noindent Now, we consider the poles of $g_{1,f}(s)$. Note that there is no pole at $s=0$. By \eqref{f1}, they are supported where $p=\frac{4\ell^2\pi^2}{M^2}$, $\ell \in \mathbb{Z}$. However, 
\begin{eqnarray*}
\sin\left(\frac{(k_2-k_1)\sqrt{p}}{2}\right)\sin\left(\frac{(M-k_1-k_2)\sqrt{p}}{2}\right)=0
\end{eqnarray*}
when $p=\dfrac{4k^2\pi^2}{(k_2-k_1)^2}$ or $p=\dfrac{4k^2\pi^2}{(M-k_2-k_1)^2}$ for $k\in\mathbb{Z}$. Thus, if for some $k, \ell\in\mathbb{Z}$
\begin{eqnarray*}
\ell=\dfrac{Mk}{k_2-k_1} \quad\text{or}\quad \ell=\dfrac{Mk}{M-k_2-k_1}
\end{eqnarray*}
which is true if and only if
\begin{eqnarray*} \label{coml}
\ell = \dfrac{\frac{M}{\text{gcd}(M, \, k_2-k_1)}k}{\frac{(k_2-k_1)}{\text{gcd}(M,\, k_2-k_1)}}\quad\text{or}\quad \ell=\dfrac{\frac{M}{\text{gcd}(M,\, M-k_2-k_1)}k}{\frac{(M-k_2-k_1)}{\text{gcd}(M,\, M-k_2-k_1)}},
\end{eqnarray*}
then $g_{1,f}(s)$ does not have poles at $p=\frac{4\ell^2\pi^2}{M^2}$. Thus, the set of poles of $g_{1,f}(s)$ is given by
\begin{eqnarray}\label{sing}
\mathcal{N}_{1,f}=\left\{\frac{4\ell^2\pi^2}{M^2} :  \ell\in\mathbb{N}, \frac{M}{\text{gcd}(M,k_2-k_1)}\nmid\ell\; \; \text{or} \;\; \frac{M}{\text{gcd}(M,M-k_2-k_1)}\nmid\ell\right\}. 
\end{eqnarray}
We next carefully calculate the residues of the poles of the integrand \eqref{int} which are given by \eqref{sing}. Take a circle $\gamma_R$ with center at the origin and of large radius $R$ which encloses a finite number of points $N$ in $\mathcal{N}_{1,f}$. Thus,

\begin{align}  \label{largeR}
& \dfrac{1}{2\pi i}\int_{\gamma_R \, \cup \, \gamma} g_{1,f}(s)g_{2,f}\left(\frac{p}{s}\right)\dfrac{ds}{s} = \sum_{\substack{p_i\in\mathcal{N}_{1,f}\\1\leq i\leq N}}\text{Res}_{s\to p_i}\left(\dfrac{g_{1,f}(s)g_{2,f}(\tfrac{p}{s})}{s}\right) \nonumber \\
&= \sideset{}{^N} \sum_{\ell \in \mathbb{N}} \dfrac{g_{2,f}\left(\frac{p M^2}{4\ell^2\pi^2}\right)}{\left(\frac{4\ell^2\pi^2}{M^2}\right)}\nonumber  \\
& \quad \quad \quad \quad \times \lim_{s\rightarrow \frac{4\ell^2\pi^2}{M^2}}\left(s-\frac{4\ell^2\pi^2}{M^2}\right)\left[-\dfrac{1}{s^{3/2}} \left(2c\cdot\dfrac{\sin\left(\frac{(k_2-k_1)\sqrt{s}}{2}\right)\sin\left(\frac{(M-k_1-k_2)\sqrt{s}}{2}\right)}{\sin\left(\frac{M\sqrt{s}}{2}\right)} + C_M\sqrt{s}\right)\right] \nonumber \\
&=-\dfrac{3\pi c}{bM^2}\sideset{}{^N}\sum_{\ell\in\mathbb{N}}\dfrac{\ell \tilde{f}(\ell)}{\left(\frac{\ell^2\pi^2}{M^2}-\frac{p}{b}\right)^{5/2}} 
\end{align}
where $\sideset{}{^N}\sum$ means that $\ell$ runs over the first $N$ elements $\frac{4\ell^2\pi^2}{M^2}$ of $\mathcal{N}_{1,f}$ and $\tilde{f}(\ell)$ is given by (\ref{ftilde}). It follows that 
\begin{eqnarray} \label{lim}
\left|\int_{\gamma_R}g_{1,f}(s)g_{2,f}\left(\frac{p}{s}\right)\dfrac{ds}{s}\right|\rightarrow 0 
\end{eqnarray}
as $R\to\infty$. Thus, \eqref{int}, \eqref{largeR}, \eqref{lim} and the contour deformation given in Figure \ref{fig:fig2} imply that
\begin{eqnarray} \label{BT}
G_f(p)=\dfrac{3\pi c}{M^2b} \sum_{\ell =1}^{\infty} \dfrac{\ell \tilde{f}(\ell)}{\left(\frac{\ell^2\pi^2}{M^2}-\frac{p}{b}\right)^{5/2}}.    
\end{eqnarray}

The set of poles of $G_f(p)$ is then
\begin{eqnarray*} \label{poles}
\mathcal{N}:=\left\{\frac{b\ell^2\pi^2}{M^2} : \ell\in\mathbb{N}, \frac{M}{\text{gcd}(M,k_2-k_1)}\nmid\ell\;\text{or}\;\frac{M}{\text{gcd}(M,M-k_2-k_1)}\nmid\ell\right\}. 
\end{eqnarray*}
Since $G_f(p)$ has two branches, by considering the cut complex plane $\mathbb{C}\setminus \mathcal{N}$, it can be made analytic in this region. This implies (1). We now compute $S^{\pm}[\mathcal{F}_f](x)$:
\begin{align}\label{eq:lateral_BL}
S^{\pm}[\mathcal{F}_f](x)=\int_0^{e^{\pm i\theta}\infty}e^{-px}\mathcal{B}[\mathcal{F}_f](p) \, dp&= C_M + \frac{3\pi c}{M^2b}\sum_{\ell =1}^{\infty} \ell \tilde{f}(\ell)\int_0^{e^{\pm i\theta}\infty} \frac{e^{-px}}{\Big(\frac{\ell^2\pi^2}{M^2}-\frac{p}{b}\Big)^{5/2}} dp
\end{align}
where the integral $\int_0^{e^{+ i\theta}\infty}$ agrees with $\int_0^{+i\infty}$ if $\Im (x)>0$ and the integral $\int_0^{e^{- i\theta}\infty}$ agrees with $\int_0^{-i\infty}$ if $\Im (x)<0$.  To prove the first part of (2), we start with the following formula, valid for $\Re(x)>0$:
\begin{equation}\label{errorf}
\int_{\gamma}\dfrac{e^{-px}}{(1-p)^{\frac{5}{2}}} \, dp=-\dfrac{4}{3}+\dfrac{8}{3}\sqrt{\pi} \mathcal{E}(\sqrt{x})    
\end{equation}
where $\gamma=[0,e^{i\theta} \infty) \cup [0,e^{-i\theta}\infty)$ and 
\begin{eqnarray*}
\mathcal{E}(y):=\dfrac{2y^3e^{-y^2}}{\sqrt{\pi}}\int_0^ye^{t^2}dt-\dfrac{y^2}{\sqrt{\pi}}.    
\end{eqnarray*}
From (\ref{smed}) and \eqref{eq:lateral_BL}, we have
\begin{align}\label{Smed}
S^{\textbf{med}}_f(x) &= C_M + \frac{1}{2} \frac{3\pi c}{M^2b}\sum_{\ell =1}^{\infty} \ell \tilde{f}(\ell) \int_{\gamma} \frac{e^{-px}}{\Big(\frac{\ell^2\pi^2}{M^2}-\frac{p}{b}\Big)^{5/2}}dp\notag\\
&= C_M + b\left(\dfrac{M}{\pi}\right)^3\dfrac{3\pi c}{2 M^2 b} \sum_{\ell =1}^{\infty} \dfrac{\tilde{f}(\ell)}{\ell} \int_{\gamma}\dfrac{e^{-\frac{bp\ell^2\pi^2 x}{M^2}}}{\left(1-p\right)^{5/2}} \, dp
\end{align}
where we have made the change of variable $p\rightarrow \frac{bp\ell^2\pi^2}{M^2}$. Thus, \eqref{errorf} and \eqref{Smed} imply
\begin{align*}
S^{\textbf{med}}_f(x) &= C_M + \dfrac{3Mc}{2 \pi^2} \sum_{\ell =1}^{\infty} \dfrac{\tilde{f}(\ell)}{\ell} \left\{-\dfrac{4}{3}+\dfrac{8}{3}\sqrt{\pi} \mathcal{E}\left(\dfrac{\ell\pi}{M}\sqrt{bx}\right)\right\}\notag\\
&=\dfrac{4Mc}{\pi^{3/2}} \sum_{\ell =1}^{\infty} \dfrac{\tilde{f}(\ell)}{\ell} \, \mathcal{E}\left(\dfrac{\ell\pi}{M}\sqrt{bx}\right)
\end{align*}
where one can check that
\begin{equation} \label{cm}
C_M = \frac{2Mc}{\pi^2} \sum_{\ell=1}^{\infty} \frac{\tilde{f}(\ell)}{\ell^2}.
\end{equation}
This implies that $S_f^{\textbf{med}}(x)$ is analytic on $\Re(x) > 0$. 

To prove the existence of the radial limit and part (3), we begin with rewriting the integral on the right-hand side of \eqref{eq:lateral_BL} as an integral of the partial theta function $\theta_{0,4M^2,\tilde{f}}^{(0)}(x)$. This requires some attention due to the branch cut singularity of the resulting square root. Integration by parts yields
\begin{align} \label{parts}
 \int_0^{e^{\pm i\theta}\infty} \frac{e^{-px}}{\Big(\frac{\ell^2\pi^2}{M^2}-\frac{p}{b}\Big)^{5/2}} dp&=
\left [ \frac{2b}{3}\frac{e^{-px}}{\Big(\frac{\ell^2\pi^2}{M^2}-\frac{p}{b}\Big)^{3/2}} \right ]_0^{e^{\pm i\theta}\infty}+\frac{2bx}{3}\int_0^{e^{\pm i\theta}\infty} \frac{e^{-px}}{\Big(\frac{\ell^2\pi^2}{M^2}-\frac{p}{b}\Big)^{3/2}} dp \notag \\
&=\frac{2b}{3} \left [-\frac{M^3}{\ell^3\pi^3}+x\int_0^{xe^{\pm i\theta}\infty} \frac{e^{-\frac{\ell^2\pi^2 b}{M^2}p}}{\Big(\frac{\ell^2\pi^2}{M^2}-\frac{\ell^2\pi^2}{M^2x}p\Big)^{3/2}} d\left( \frac{\ell^2\pi^2b}{M^2x}p \right) \right ] \notag \\
&= -\frac{2bM^3}{3\ell^3\pi^3} + \frac{2b^2 M}{3 \ell \pi} \int_0^{xe^{\pm i\theta}\infty} \frac{e^{-\frac{\ell^2\pi^2b}{M^2}p}}{\Big(1-\frac{p}{x}\Big)^{3/2}} dp.
\end{align}
The change of coordinate $x \rightarrow -\frac{1}{2\pi i\alpha}$ where $0 \neq \alpha \in \mathbb{C}$ to (\ref{parts}) yields
\begin{align} \label{change}
 -\frac{2bM^3}{3\ell^3\pi^3}  + \frac{2b^2 M}{3 \ell \pi} \int_0^{xe^{\pm i\theta}\infty} \frac{e^{-\frac{\ell^2\pi^2b}{M^2}p}}{\Big(1-\frac{p}{x}\Big)^{3/2}} dp 
&=-\frac{2bM^3}{3\ell^3\pi^3} + \frac{2b^2 M}{3 \ell \pi} \int_0^{-\frac{1}{2\pi i\alpha}e^{\pm i\theta}\infty} \frac{e^{2\pi i\frac{\ell^2 2\pi i b}{4M^2}p}}{\big(1+2\pi i p\alpha\big)^{3/2}} dp \notag \\
&=-\frac{2bM^3}{3\ell^3\pi^3} + \frac{b^2M}{3i\ell \pi^2} \int_0^{-\frac{1}{\alpha}e^{\pm i\theta}\infty} \frac{e^{2\pi i\frac{\ell^2  b}{4M^2}p}}{\big(1+ p\alpha\big)^{3/2}} dp.
\end{align}
Now $(1+p\alpha)^{3/2}$ has a branch cut at $p=-1/\alpha$, hence we distinguish between $\Re(\alpha)>0$ and $\Re(\alpha)<0$ (see Figure~\ref{fig:branch}). If $\Re(\alpha)>0$, then $(1+p\alpha)^{3/2}=\alpha^{3/2}(1/\alpha+p)^{3/2}$, and the integral at angle $-i\theta$ will not change sign as the path is not crossing the branch cut, i.e., $\int_0^{-\frac{1}{\alpha}e^{ -i\theta}\infty}=\int_0^{+i\infty}$. If $\Re(\alpha)<0$, then $(1+p\alpha)^{3/2}=\alpha^{3/2}(1/\alpha+p)^{3/2}$, but the integral at angle $+i\theta$ will change sign as the path is crossing the branch cut, i.e., $\int_0^{-\frac{1}{\alpha}e^{+i\theta}\infty}=-\int_0^{+i\infty}$.
\begin{figure}[ht]
\center
\begin{tikzpicture}
\draw (-7,0)--(-3,0);
\draw (-5,-2)--(-5,2);
\draw[thick] (-5,0)--(-3,0.5) node[pos=0.5,  
    xscale=-1,
    sloped] {\tiny{$\blacktriangleleft$}};;
\draw[thick] (-5,0)--(-3,2) node[pos=0.5,  
    xscale=-1,
    sloped] {\tiny{$\blacktriangleleft$}};
\draw[decorate, decoration={snake}] (-4.5,0.3)--(-7,0.3) node[above,pos=-0.2]{$-\frac{1}{\alpha}$};
\node at (-4.45,0.3) {$\bullet$};
\draw (-2,0)--(2,0);
\draw (0,-2)--(0,2);
\draw[thick] (0,0)--(-2,0.5) node[pos=0.8,  
    xscale=-1,
    sloped] {\tiny{$\blacktriangleright$}};
    \draw[thick] (0,0)--(-2,2) node[pos=0.5,  
    xscale=-1,
    sloped] {\tiny{$\blacktriangleright$}};
\draw[decorate, decoration={snake}] (-0.55,0.3)--(-2,0.3) node[above,pos=-0.2]{$-\frac{1}{\alpha}$};
\node at (-0.55,0.3) {$\bullet$};
\end{tikzpicture}
\caption{Branch cut of $(1+p\alpha)^{3/2}$. On the left for $\Re(\alpha)<0$ and on the right for $\Re(\alpha)>0$.}
\label{fig:branch}
\end{figure}

\noindent By \eqref{eq:lateral_BL}, \eqref{cm} and \eqref{change}, we have for $\alpha\in\Q_{>0}$
\begin{align*}
\int_0^{e^{-i\theta}\infty} e^{\frac{p}{2\pi i\alpha}}\, \mathcal{B}[\mathcal{F}_f](p) \,dp &=-i\frac{c b }{M\pi}\sum_{\ell =1}^{\infty}\tilde{f}(\ell) \int_0^{-\frac{1}{\alpha}e^{- i\theta}\infty} \frac{e^{2\pi i\frac{\ell^2  b}{4M^2}p}}{\big(1+ p\alpha\big)^{3/2}} dp \nonumber \\
&=-i\frac{ c b }{M\pi \alpha^{3/2}}\sum_{\ell =1}^{\infty} \tilde{f}(\ell) \int_0^{+i\infty} \frac{e^{2\pi i\frac{\ell^2  b}{4M^2}p}}{\big(1/\alpha+ p\big)^{3/2}} dp \nonumber \\
&=-i\frac{c b }{M\pi \alpha^{3/2}}\int_0^{+i\infty} \frac{\theta_{0,4M^2,\tilde{f}}^{(0)}(bp)}{\big(1/\alpha+ p\big)^{3/2}} dp
\end{align*}
while for $\alpha\in\Q_{<0}$, we obtain
\begin{align*}
\int_0^{e^{+ i\theta}\infty} e^{\frac{p}{2\pi i\alpha}}\, \mathcal{B}[\mathcal{F}_f](p) \,dp &=-i\frac{c b }{M\pi}\sum_{\ell =1}^{\infty} \tilde{f}(\ell) \int_0^{-\frac{1}{\alpha}e^{+i\theta}\infty} \frac{e^{2\pi i\frac{\ell^2  b}{4M^2}p}}{\big(1+ p\alpha\big)^{3/2}} dp \nonumber \\
&= i\frac{c b }{M\pi \alpha^{3/2}} \sum_{\ell=1}^{\infty} \tilde{f}(\ell) \int_0^{+i\infty} \frac{e^{2\pi i\frac{\ell^2  b}{4M^2}p}}{\big(1/\alpha+ p\big)^{3/2}} dp \nonumber \\
&=i\frac{c b }{M\pi \alpha^{3/2}}\int_0^{+i\infty} \frac{\theta_{0,4M^2,\tilde{f}}^{(0)}(bp)}{\big(1/\alpha+ p\big)^{3/2}} dp.
\end{align*}
We now compute the discontinuity $\mathrm{disc}^0(x)$ across the positive real axis:
\begin{equation} \label{disc}
\begin{aligned}
\mathrm{disc}^0(x) &:= S^{+}[\mathcal{F}_f](x) - S^{-}[\mathcal{F}_f](x)  \\
&=\frac{3\pi c}{M^2b}\sum_{\ell =1}^{\infty} \ell \tilde{f}(\ell) \left[\int_0^{e^{+i\theta}\infty}\frac{e^{-px}}{\big(\frac{\ell^2\pi^2}{M^2}-\frac{p}{b}\big)^{5/2}}dp - \int_0^{e^{-i\theta}\infty}\frac{e^{-px}}{\big(\frac{\ell^2\pi^2}{M^2}-\frac{p}{b}\big)^{5/2}}dp\right] \\
&= \frac{3\pi c}{M^2b}\sum_{\ell =1}^{\infty} \ell \tilde{f}(\ell) \int_{\mathcal{C}} \frac{e^{-px}}{\big(\frac{\ell^2\pi^2}{M^2}-\frac{p}{b}\big)^{5/2}}dp \\
&= \frac{3\pi cb^{3/2}}{M^2}\sum_{\ell =1}^{\infty} \ell \tilde{f}(\ell) e^{-\frac{\ell^2\pi^2 b}{M^2}x} \int_{\mathcal{C}_{\ell}} \frac{e^{-px}}{(-p)^{5/2}}\, dp \\
&=\frac{3\pi cb^{3/2}\,}{M^2} \sum_{\ell =1}^{\infty} \ell \tilde{f}(\ell) e^{-\frac{\ell^2\pi^2 b}{M^2}x} \int_{\mathcal{C}_{\ell}} \frac{e^{-px}}{(-xp)^{5/2}}\, d(xp) x^{3/2} \\
&=\frac{2\pi i}{\Gamma\big(\frac{5}{2}\big)}x^{3/2}\frac{3\pi cb^{3/2}}{M^2} \sum_{\ell =1}^{\infty}  \ell \tilde{f}(\ell) e^{-\frac{\ell^2\pi^2 b}{M^2}x} \\
&=2i\, (2b \pi x)^{3/2}\frac{ \sqrt{2} c }{M^2} \sum_{\ell=1}^{\infty}  \ell \tilde{f}(\ell) e^{-\frac{\ell^2\pi^2 b}{M^2}x} \\
& = 2i\, (2b \pi x)^{3/2}\frac{ \sqrt{2} c }{M^2} \theta^{(1)}_{0,4M^2,\tilde{f}}(2\pi ix). 
\end{aligned}
\end{equation}
Here,  $\mathcal{C}_\ell=\mathcal{C}-\tfrac{\ell^2\pi^2b}{M^2}$ where $\mathcal{C}$ is the Hankel contour enclosing the branch point singularity arising from the second line in \eqref{disc}, as in Figure \ref{fig:Hankel}. The radial limit $x \to -\frac{1}{2\pi i\alpha}$ of the partial theta series appearing in the last line of \eqref{disc} exists by \cite[Proposition 3.1]{akl} and so
\begin{equation*}
\mathrm{disc}^0\Big(-\frac{1}{2\pi i\alpha}\Big) = 2i\, \Big(\frac{ib}{\alpha} \Big)^{3/2}\frac{ \sqrt{2} c }{M^2}\, \theta_{0,4M^2,\tilde{f}}^{(1)}\left(-\frac{b}{\alpha}\right).
\end{equation*}

\begin{figure}[ht]
\center
\begin{tikzpicture}
\draw[->] (-7,0)--(-3,0);
\draw[->] (-5,-2)--(-5,2);
\node[red] at (-4.5,0) {$\bullet$};
\node[red] at (-4,0) {$\bullet$};
\node[red] at (-3.5,0) {$\bullet$};
\draw[thick,color={rgb, 255:red, 80; green, 227; blue, 194 }  ,opacity=1] (-5,0)--(-3,0.5) node[pos=0.5,  
    xscale=-1,
    sloped] {$<$};;
\draw[thick,color={rgb, 255:red, 74; green, 144; blue, 226 }  ,opacity=1] (-5,0)--(-3,-0.5) node[pos=0.5,  
    xscale=-1,
    sloped] {$<$};
\node[above,font=\small] at (-3,0.5) {$S^{+}[\mathcal{F}_f]$};
\node[below,font=\small] at (-3,-0.5) {$S^{-}[\mathcal{F}_f]$};
\draw[->] (-2,0)--(2,0);
\draw[->] (0,-2)--(0,2);
\node[red] at (0.5,0) {$\bullet$};
\node[red] at (1,0) {$\bullet$};
\node[red] at (1.5,0) {$\bullet$};
\draw[thick] (0,0.5)--(2,0.5) node[pos=0.5,  
    xscale=-1,
    sloped] {$<$};;
\draw[thick] (0,-0.5)--(2,-0.5) node[pos=0.5,  
    xscale=-1,
    sloped] {$>$};
 \draw[thick] (0,0.5) arc (90:270:0.5); 
 \node[font=\small] at (2,0.75) {$\mathcal{C}$};  
\end{tikzpicture}
\caption{The Hankel contour $\mathcal{C}$ for the discontinuity $\mathrm{disc}^0(x)$.}
\label{fig:Hankel}
\end{figure}
\noindent In addition, the radial limit $x \to -\frac{1}{2\pi i\alpha}$ also exists for $S^{\pm}[\mathcal{F}_f](x)$ since they are analytic functions on the half-planes $\Re(e^{\pm i\theta}x)>0$, respectively (see the change of coordinates after \eqref{parts}). 
Thus, the second part of (2) follows. Summarizing, we find that 
\[
{S}^{\mathbf{med}}_f\left(-\frac{1}{2\pi i\alpha}\right)
=\begin{cases}
S^{+}\left(-\dfrac{1}{2\pi i\alpha}\right)+\, \Big(\dfrac{b}{\alpha} \Big)^{3/2} e^{\pi i/4}\dfrac{ \sqrt{2} c }{M^2}\, \theta_{0,4M^2,\tilde{f}}^{(1)}\left(-\frac{b}{\alpha}\right)
 & \text{if $\alpha\in\Q_{>0}$}, \\
&\\
S^{-}\left(-\dfrac{1}{2\pi i\alpha}\right)-\, \Big(\dfrac{b}{\alpha} \Big)^{3/2} e^{\pi i/4} \dfrac{ \sqrt{2} c }{M^2}\, \theta_{0,4M^2,\tilde{f}}^{(1)}\left(-\frac{b}{\alpha}\right)
& \text{if $\alpha\in\Q_{<0}$}
\end{cases}
\]  
with
\[{S}^\pm\left(-\frac{1}{2\pi i\alpha}\right)=\mp i\frac{c b}{M\pi \alpha^{3/2}}\int_0^{+i\infty} \frac{\theta_{0,4M^2,\tilde{f}}^{(0)}(bp)}{\big(1/\alpha+ p\big)^{3/2}} dp.\]
Finally, noticing that $(i\alpha)^{-3/2}=\mp e^{\pi i/4}\alpha^{-3/2}$ depending on the sign of $\alpha$ and on our choice of the branch of the square root, part (3) follows.
\end{proof}

\section{Proof of Theorem~\ref{main2} and Corollary~\ref{cor}} \label{sec:torus}
In order to prove Theorem~\ref{main2} and Corollary~\ref{cor}, we require some general preliminaries. Recall the periodic function $\chi_{s,t}^{(n,m)}(k)$ given by \eqref{pf}. 

\begin{lemma}\label{scha}
For all $k\in\mathbb{Z}$, $1\le m\le t-1$ and $1\le n\le s-1$, we have
\begin{align*}
\chi_{s,t}^{(n,m)}(k)=\chi_{s,t}^{(s-n,t-m)}(k).    
\end{align*}
\end{lemma}
\begin{proof}
By replacing $n\rightarrow s-n$ and $m\rightarrow t-m$, we see that
\begin{equation}\label{sym}
\begin{aligned}
\pm ((s-n)t-(t-m)s)&=\pm(ms-nt)=\mp(nt-ms)\,,\\
\pm ((s-n)t+(t-m)s)&=\pm(2st-ms-nt)\equiv \mp(nt+ms)\pmod{2st}\,,     
\end{aligned}
\end{equation}
and thus the result follows from \eqref{sym}.
\end{proof}
In view of Lemma~\ref{scha}, the map 
\begin{align} \label{bi}
(n, m)\xrightarrow[]{} \begin{cases}
(n, m)& \text{if $1\le n\le \frac{s-1}{2}, \quad 1\le m\le \frac{t-1}{2}$},\\
(s-n, t-m) &\text{otherwise}.
\end{cases}    
\end{align}
is a bijection between the pairs defined by $\mathscr{D}_1(s, t)$ and $\mathscr{D}_2(s, t)$ in \eqref{d1} and \eqref{d2}, respectively.

\begin{lemma}\label{distinct}
Let $s$ and $t$ be coprime positive integers and $\mathscr{D}(s,t)$ be the set given by (\ref{d}). Then the integers in the set 
\begin{equation*} \label{set1}
\mathscr{S} = \{\pm(nt\pm ms): (n,m)\in\mathscr{D}(s,t)\}
\end{equation*} 
are distinct and thus $| \mathscr{S} | = 2(s-1)(t-1)$.    
\end{lemma}
\begin{proof}
We assume that $s$ and $t$ are coprime odd integers. A similar argument applies when $s$ and $t$ are of opposite parities. Let us consider two integers of the forms $n_1t+m_1s$ and $n_2t+m_2s$ where $-\frac{s-1}{2}\le n_1, n_2\leq \frac{s-1}{2}$ and $-(t-1)\le m_1, m_2\leq t-1$. Then 
\begin{align}\label{liceq}
n_1t+m_1s=n_2t+m_2s\iff (n_1-n_2)t+(m_1-m_2)s=0.    
\end{align}
Since $s$ and $t$ are coprime, \eqref{liceq} implies that $s\mid (n_1-n_2)$, which yields $|n_1-n_2|\ge s$. On the other hand, since $-\frac{s-1}{2}\le n_1, n_2\leq \frac{s-1}{2}$, we obtain $-(s-1)\leq n_1-n_2\leq s-1$, or equivalently, $|n_1-n_2|\le s-1$, a contradiction. This shows that $n_1=n_2$ and thus \eqref{liceq} implies that $m_1=m_2$.
\end{proof}

Define
\begin{align*}
S_{n, \, m}^{n',\,m'}=\sqrt{\dfrac{8}{st}}(-1)^{nm'+mn'+1}\sin\left(\frac{nn't}{s}\pi\right)\sin\left(\frac{mm's}{t}\pi\right).    
\end{align*}

\begin{proposition}\label{gentor}
For all $k\in\mathbb{Z}$, and $(n, m)\in \mathscr{D}(s, t)$, we have
\begin{equation} \label{keyform}
\tilde{\chi}_{s,t}^{(n,m)}(k)=-\sqrt{\dfrac{st}{8}}\sum_{(n',\,m')\in \mathscr{D}(s, t)} S_{n, \, m}^{n',\,m'} \chi_{s,t}^{(n',\,m')}(k).    
\end{equation}
\end{proposition}
\begin{proof}
We assume that $nt - ms > 0$. A similar argument holds for $nt - ms < 0$. By \eqref{ftilde}, we have 
\begin{align} \label{tilchiano}
\tilde{\chi}_{s,t}^{(n,m)}(k)&=(-1)^k\sin\left(\dfrac{(nt+ms- |nt-ms|)k\pi}{2st}\right)\sin\left(\dfrac{(2st- |nt-ms| - (nt+ms))k\pi}{2st}\right) \nonumber \\   
& =(-1)^k\sin\left(\dfrac{mk\pi}{t}\right)\sin\left(\dfrac{(s-n)k\pi}{s}\right) \nonumber \\
& =-\sin\left(\frac{mk\pi}{t}\right)\sin\left(\frac{nk\pi}{s}\right).
\end{align}
Clearly, the support of $\tilde{\chi}_{s,t}^{(n,m)}$ is contained within the set of integers $k$ such that $s\nmid k$ and $t\nmid k$. By inclusion-exclusion, the total number of integers between $0$ and $2st-1$ which are either multiples of $s$ or $t$ is $2s+2t-2$. By Lemma \ref{distinct}, the number of distinct integers $k=\pm(m's\pm n't)$ between $0$ and $2st-1$, where $(n',m')\in \mathscr{D}(s, t)$ is $2(s-1)(t-1)$. Since $\tilde{\chi}_{s,t}^{(n,m)}$ has period $2st$, this implies that the support of $\tilde{\chi}_{s,t}^{(n,m)}$ is contained in the following set of distinct integers
\begin{equation}\label{intgen}
\mathscr{T} = \{ \pm(m's\pm n't)\pmod{2st} \, : \, (n',m')\in\mathscr{D}(s,t)\}.    
\end{equation}
Thus it suffices to prove \eqref{keyform} for integers in $\mathscr{T}$. From \eqref{tilchiano}, we have
\begin{align}\label{tilchigenlin}
\tilde{\chi}_{s,t}^{(n,m)}(\pm(m's\pm n't))&=-\sin\left(\frac{\pm m(m's\pm n't)\pi}{t}\right)\sin\left(\frac{\pm n(m's\pm n't)\pi}{s}\right)\notag\\
&=-\sin\left(\pm mn'\pi+ \frac{mm's}{t}\pi\right)\sin\left( nm'\pi\pm \frac{ nn't}{s}\pi\right) \notag \\
&=\pm(-1)^{mn'+nm'+1}\sin\left(\frac{mm's}{t}\pi\right)\sin\left(\frac{nn't}{s}\pi\right) \notag\\
&=(-1)^{mn'+nm'}\sin\left(\frac{mm's}{t}\pi\right)\sin\left(\frac{nn't}{s}\pi\right)\chi_{s,t}^{(n',m')}(\pm(m's\pm n't))
\end{align}
where in the last step we use that the sign $\pm$ depends on the sign of $(m's\pm n't)$, and it is $-\chi_{s,t}^{(n',m')}(\pm(m's\pm n't))$. Since the integers in $\mathscr{T}$ are distinct, then from \eqref{tilchigenlin} we find:
\begin{align}\label{classgen}
\tilde{\chi}_{s,t}^{(n,m)}(\pm(m's\pm n't))=-\sqrt{\dfrac{st}{8}}\sum_{(n'',\,m'')\in\mathscr{D}(s,t)}S_{n, \, m}^{n'',\,m''}\chi_{s,t}^{(n'',\,m'')}(\pm(m's\pm n't)).    
\end{align}
It follows from \eqref{classgen} that \eqref{keyform} is true for integers $k$ given by \eqref{intgen}, which are non-multiples of $s$ or $t$ and between $0$ and $2st-1$. It remains to show that \eqref{keyform} remains true when $k$ is a multiple of $s$ or $t$. Without loss of generality, we assume that $k$ is a multiple of $s$. Then \eqref{tilchiano} implies that $\tilde{\chi}_{s,t}^{(n,m)}(k)=0$. Thus, it suffices to show that $\chi_{s,t}^{(n',\,m')}(k)=0$ for all $(n', m')\in\mathscr{D}(s, t)$. To this end, we note that since $s$ and $t$ are coprime, $s\mid \pm (n't\pm m's)$ if and only if $s\mid n'$, which implies that $n'\ge s$, a contradiction. Thus, for $(n', m')\in\mathscr{D}(s, t)$, it follows $\pm(n't\pm m's)\not\equiv 0\pmod{s}$. Hence $k\neq \pm(n't\pm m's)$, and  $\chi_{s,t}^{(n',m')}(k)=0$.
\end{proof}

Finally, we recall two key properties for the non-holomorphic Eichler integral
\begin{equation*}
\widehat{\Phi}^{(n,m)}(z) := \sqrt{\frac{sti}{8\pi^2}} \int_{\bar{z}}^{+i \infty} \frac{\theta_{0,4st,\chi_{s,t}^{(n,m)}}^{(0)}(\tau)}{(\tau - z)^{3/2}} d\tau 
\end{equation*}
where $z \in \mathbb{H}^{-}$ \cite[Eqn.~(16)]{HK1}. For $z \in \mathbb{H}^{-}$ and $\alpha \in \mathbb{Q}$, consider the period function
\begin{equation*}
r^{(n,m)}(z; \alpha) := \sqrt{\frac{sti}{8\pi^2}} \int_{\alpha}^{+i \infty} \frac{\theta_{0,4st,\chi_{s,t}^{(n,m)}}^{(0)}(\tau)}{(\tau - z)^{3/2}} d\tau\,, 
\end{equation*}
as defined in \cite[Eqn.~(18)]{HK1}. Then, we have
\begin{equation} \label{HKlast1}
\widehat{\Phi}^{(n,m)}(z) + \left(\frac{1}{iz}\right)^{3/2} \sum_{(n',\,m')\in\mathscr{D}(s,t)} S_{n, m}^{n', \,m'} \widehat{\Phi}^{(n',\,m')}\left(-\frac{1}{z}\right) = r^{(n,m)}(z;0)
\end{equation}
and
\begin{equation} \label{HKlast2}
\widehat{\Phi}^{(n,m)}(\alpha) =  -\frac{1}{2} \theta_{0,4st,\chi_{s,t}^{(n,m)}}^{(1)}(\alpha)
\end{equation}
where \eqref{HKlast1} is \cite[Eqn.~(17)]{HK1} and \eqref{HKlast2} is \cite[below Eqn.~(19)]{HK1}. 

\begin{proof}[Proof of Theorem \ref{main2}]
By Proposition~\ref{gentor}, we have
\begin{equation}\label{eq:gentor}
\theta_{0, 4(2st)^2, \tilde{\chi}_{s,t}^{(n,m)}}^{(\nu)}(4st z) = -\sqrt{\dfrac{st}{8}} \sum_{(n',m')\in\mathscr{D}(s,t)} S_{n, m}^{n', m'} \theta_{0, 4st, \chi_{s,t}^{(n',m')}}^{(\nu)}(z).
\end{equation}
Also, \cite[Eqn.~(10)]{HK1} states
\begin{equation} \label{transgen}
\theta_{0,4st,\chi_{s,t}^{(n',m')}}^{(0)}(z) = \sqrt{\frac{i}{z}} \sum_{(n',\,m')\in\mathscr{D}(s,t)} S_{n, m}^{n', \,m'}\theta_{0,4st,\chi_{s,t}^{(n',m')}}^{(0)}\left(-\frac{1}{z} \right)\,.
\end{equation}
Applying Theorem~\ref{main} part~(3) and \eqref{eq:gentor} gives
\begin{multline} \label{compgen}
\displaystyle S_{\chi_{s,t}^{(n,m)}}^{\bf{med}} \left(-\frac{1}{2 \pi i \alpha} \right)  = -\frac{2e^{\pi i/4}}{ \pi (i \alpha)^{3/2}}\sqrt{\dfrac{st}{8}}  \int_{0}^{+ i\infty} \frac{\sum_{(n',m')\in\mathscr{D}(s,t)} S_{n, m}^{n', m'}\theta_{0, 4st, \chi_{s,t}^{(n',m')}}^{(0)}(p)}{\left( \frac{1}{\alpha} + p \right)^{3/2}} \, dp \\
 - \frac{1}{(i \alpha)^{3/2}}\sum_{(n',m')\in\mathscr{D}(s,t)} S_{n, m}^{n', m'}\theta_{0, 4st, \chi_{s,t}^{(n',m')}}^{(1)}\left(-\dfrac{1}{\alpha}\right).
 \end{multline}
Then, by \eqref{transgen}, we rewrite \eqref{compgen} as follows:
\begin{multline} \label{almost}
\displaystyle S_{\chi_{s,t}^{(n,m)}}^{\bf{med}} \left(-\frac{1}{2 \pi i \alpha} \right) = -\frac{2e^{\pi i/4}}{ \pi (i \alpha)^{3/2}}\sqrt{\dfrac{sti}{8}}  \int_{0}^{+ i \infty} \frac{\theta_{0, 4st, \chi_{s,t}^{(n,m)}}^{(0)}(-1/p)}{\sqrt{p} \left(\frac{1}{\alpha} + p \right)^{3/2}} \, dp \\
 - (i \alpha)^{-3/2} \sum_{(n',m')\in\mathscr{D}(s,t)} S_{n, m}^{n', m'}\theta_{0, 4st, \chi_{s,t}^{(n',m')}}^{(1)}\left(-\dfrac{1}{\alpha}\right).
\end{multline}
We now perform the change of coordinates $p \to -1/p$ to the integral in \eqref{almost} and simplify to obtain
\begin{align*}
\frac{2e^{\pi i/4}}{ \pi (i \alpha)^{3/2}}\sqrt{\dfrac{sti}{8}}\int_{0}^{+ i \infty} \frac{\theta_{0, 4st, \chi_{s,t}^{(n,m)}}^{(0)}(-1/p)}{\sqrt{p}\left(\frac{1}{\alpha} + p \right)^{3/2}} \, dp
=2\sqrt{\dfrac{sti}{8\pi^2}}\int_{0}^{+ i \infty} \frac{\theta_{0, 4st, \chi_{s,t}^{(n,m)}}^{(0)}(p)}{\left(p - \alpha \right)^{3/2}} \, dp\,.
\end{align*}
Then, going back to \eqref{almost} we find
\begin{equation} \label{lastgen}
\begin{aligned}
\displaystyle S_{\chi_{s,t}^{(n,m)}}^{\bf{med}} \left(-\frac{1}{2 \pi i \alpha} \right) &= -2\sqrt{\dfrac{sti}{8\pi^2}}\int_{0}^{+ i \infty} \frac{\theta_{0, 4st, \chi_{s,t}^{(n,m)}}^{(0)}(p)}{\left(p - \alpha \right)^{3/2}} \, dp \\
&\qquad - (i \alpha)^{-3/2} \sum_{(n',m')\in\mathscr{D}(s,t)} S_{n, m}^{n', m'}\theta_{0, 4st, \chi_{s,t}^{(n',m')}}^{(1)}\left(-\dfrac{1}{\alpha}\right) \\
& = \theta_{0, 4st, \chi_{s,t}^{(n,m)}}^{(1)}(\alpha)
\end{aligned}
\end{equation}
where the last line of (\ref{lastgen}) follows from \eqref{HKlast1} and \eqref{HKlast2}. This proves the result.
\end{proof}

\begin{proof}[Proof of Corollary \ref{cor}]
Let $u \in \mathbb{Z}_{\geq 1}$. For $0 \leq \ell \leq u-1$, define
\begin{eqnarray} \label{xm}
X_u^{(\ell)}(q):=\sum_{k_1,k_2,\dotsc,k_u=0}^\infty (q)_{k_u} q^{k_1^2+\cdots+k_{u-1}^2+k_{\ell+1}+\cdots+k_{u-1}}\prod_{i=1}^{u-1}\begin{bmatrix} k_{i+1} + \delta_{i,\ell} \\ k_{i} \end{bmatrix}
\end{eqnarray}
where $\delta_{i,\ell}$ is the characteristic function; see~\cite[Eqn.~(11)]{hikami2}. The expression $X_u^{(\ell)}(q)$ matches the $N$th colored Jones polynomial for $T(2, 2u+1)$ when $\ell = 0$ and $q=\zeta_N$ and is an element of $\mathcal{H}$; see \cite[Proposition 16]{hikami2}. Take $s=2$, $t=2u+1$, $n=1$ and $m=\ell + 1$ in \eqref{pf}. Then Hikami's strange identity~\cite[Eqn. (15)]{hikami2}\footnote{Taking $u=1$ and $\ell = 0$ in (\ref{xm}) and (\ref{strangehikami}) yields (\ref{strid}).} reads
\begin{eqnarray}\label{strangehikami}
X_u^{(\ell)}(q)``=" -\frac{1}{2}\theta^{(1)}_{(2u-2\ell-1)^2, 2(8u+4), \chi_{2, 2u+1}^{(1, \ell +1)}}(q).
\end{eqnarray} 
Claims~(1) and~(2) of Conjecture~\ref{cgconj} follow from Theorem~\ref{main} parts~(1) and~(2), namely the formal series $\mathcal{F}_{\chi_{2, 2u+1}^{(1, \ell + 1)}}(x)$ has a resurgent Borel transform and its median resummation is an analytic function for $\Re(x)>0$ with radial limits at points $\frac{1}{2\pi i}\mathbb{Q}$. In particular, this is true for \[F_{T(2,2u+1)}(x) = \mathcal{F}_{\chi_{2, 2u+1}^{(1,1)}}(x)\,.\] Then, part~(3) of Conjecture~\ref{cgconj} follows by Theorem~\ref{main2} applying \eqref{strangehikami} for $\ell=0$. In fact, we can prove a stronger result:
\begin{equation}
S^{\bf med}_{\chi_{2, 2u+1}^{(1, \ell + 1)}}\left(-\frac{1}{2\pi i\alpha}\right) \stackrel{.}{=} \mathcal{F}_{\chi_{2, 2u+1}^{(1,\ell+1)}}(e^{2\pi i \alpha})
\end{equation}
for every $0\neq\alpha\in\mathbb{Q}$ and $0\leq\ell\leq u-1$.

Let $k\in\mathbb{Z}_{\geq 1}$ and $\mathfrak{F}_k(q)$ denote the element in $\mathcal{H}$ that matches the $N$th colored Jones polynomial for $T(3, 2^k)$ at a root of unity $q=e^{\frac{2\pi i}{N}}$ (see~\cite[Eqn.~(1.8)]{Be} for an explicit $q$-hypergeometric expression). Choose $s=3$, $t=2^k$, $n=2$ and $m=1$ in \eqref{pf}. Then the strange identity 
\begin{equation} \label{str2}
\mathfrak{F}_k(q)``=" -\frac{1}{2}\theta^{(1)}_{(2^{k+1} - 3)^2, 3\cdot 2^{k+2}, \chi_{3, 2^{k}}^{(2,1)}}(q) 
\end{equation}
was proved in~\cite[Theorem~2.4]{Be}\footnote{Taking $k=1$ in (\ref{str2}) gives (\ref{strid}).}. 
Claims~(1) and~(2) of Conjecture~\ref{cgconj} follow from Theorem~\ref{main} parts~(1) and~(2), namely the formal series \[F_{T(3,2^k)}(x) = \mathcal{F}_{\chi_{3, 2^{k}}^{(2,1)}}(x)\] has a resurgent Borel transform and its median resummation is an analytic function for $\Re(x)>0$ with radial limits at points $\frac{1}{2\pi i}\mathbb{Q}$. Then, part~(3) of Conjecture~\ref{cgconj} follows by Theorem~\ref{main2} applying \eqref{str2}.
\end{proof}

\section{Concluding Remarks}\label{sec:remarks}
In order to apply Theorems~\ref{main} and~\ref{main2} and thus verify Conjecture~\ref{cgconj} for all torus knots $T(s,t)$, one needs to prove the relevant strange identity in \eqref{Hstrange}. The first obstruction in this task is finding an explicit ``non-cyclotomic" expansion for $J_N(T(s,t); q)$ from which an element in $\mathcal{H}$ such as the one in \eqref{xm} can be extracted. The Rosso-Jones formula for $J_N(T(s,t); q)$ does not appear to be sufficient~\cite[page 132]{morton}. Instead, one should consider the walks along braids method in~\cite{armond, bs}. The second obstruction is in determining the underlying $q$-series identity which implies the strange identity. Thus, it would be of substantial interest to further develop the techniques in~\cite{lovejoy}. Finally, what can one say about Conjecture~\ref{cgconj} for satellite or hyperbolic knots?

\section*{Acknowledgements}
The first author thanks Dan Zhang, Nick Dorey and Phil Trinh for many interesting discussions. He was supported by EPSRC grant EP/V012479/1 as a postdoc at the University of Bath while part of this work was completed. The fourth author was partially supported by Enterprise Ireland CS20212030 and the Irish Research Council Advanced Laureate Award IRCLA/2023/1934. He would like to thank the Okinawa Institute of Science and Technology for their hospitality and support during his visit from June 9 to September 1, 2022 as part of their Theoretical Sciences Visiting Program. In particular, the first and fourth authors are very grateful to Professor Reiko Toriumi (OIST) for her firm instance for us to converse after the first author's talk at OIST on August 23, 2022. The fourth author would also like to thank the Max-Planck-Institut f\"ur Mathematik for their hospitality and support during the completion of this paper. The fifth author has been supported by the Huawei Young Talents Program at IH\'ES. Finally, the authors thank David Sauzin for his helpful comments on median resummation and the referees for their critiques which substantially improved the paper.

\end{document}